\documentclass[10pt,reqno]{amsart}

\usepackage{amssymb, verbatim, enumerate}
\usepackage{amsmath}
\usepackage{graphicx}
\usepackage{epstopdf,fancybox,color}
\usepackage[mathscr]{eucal}

\newtheorem{theorem}{Theorem}

\newtheorem{lemma}[theorem]{Lemma}

\theoremstyle{definition}
\newtheorem{definition}[theorem]{Definition}
\newtheorem{example}{Example}

\theoremstyle{remark}

\title[Chlodowsky variant of Bernstein-type operators on the domain]{Chlodowsky variant of Bernstein-type operators on the domain}

\author[S. K{\i}rc{\i} Serenbay]{Sevilay K{\i}rc{\i} Serenbay}
\address[S. K{\i}rc{\i} Serenbay]{ Department of Mathematics,
              Faculty of Science,
              Harran University, \c{S}anl{\i}urfa; Department of Mathematics,
              Faculty of Science,
              Ankara University,
              06100 Tando\u{g}an, Ankara, Turkey}
\email[S. K{\i}rc{\i} Serenbay]{skserenbay@harran.edu.tr,skserenbay@ankara.edu.tr}

\author[R. Akta\c{s} Karaman]{Rabia Akta\c{s} Karaman}
\address[R. Akta\c{s} Karaman]{Department of Mathematics,
              Faculty of Science,
              Ankara University,
              06100 Tando\u{g}an, Ankara, Turkey}
\email[R. Akta\c{s} Karaman (corresponding author)]{raktas@science.ankara.edu.tr}

\date{\today}


\date{\today}

\begin{document}

\begin{abstract}
In the present paper, we deal with Bernstein-Chlodowsky type operators for
approximating functions on the domain. We first present Bernstein-Chlodowsky
type operators in two variables and then we discuss some examples of these
operators under a domain transformation.\textbf{\ }Finally, we give bivarite shifted $m$th
Bernstein-Chlodowsky-Stancu operators and we present some figures for
approximation properties of our operator.

\end{abstract}

\maketitle

\section{Introduction}
The classical Bernstein operators were initially introduced by S. Bernstein \cite{bernstein} in 1912 as a way to approximate continuous functions on the interval $[0,1]$ using polynomials as
\[
B_{m}\left(  g,x\right)  =%
{\displaystyle \sum \limits_{k=0}^{m}}
g\left(  \frac{k}{m}\right)  p_{m,k}\left(  x\right)  ,
\]
where $p_{m,k}\left(  x\right)  =\binom{m}{k}x^{k}\left(  1-x\right)  ^{m-k}$
for $0\leq x\leq1,~m\geq0$,
and they are well-known for their convergence properties. Bernstein operators are extensively used in approximating continuous functions. They offer a constructive proof of the Weierstrass approximation theorem \cite{korovkin}, showing that any continuous function on a closed interval can be uniformly approximated by a sequence of
Bernstein polynomials. Over the years, numerous modifications and extensions
of Bernstein operators have been studied by many authors, refer to
\cite{butzer,Kantorovitch,CGM,CGR,Derriennic,Derriennic2}. Bernstein polynomials have several
applications in approximation theory \cite{korovkin,Lorentz}. Besides this, Bernstein
polynomials form the basis of B\'{e}zier curves \cite{Bezier} and surfaces,
which are fundamental tools in computer graphics, animation, and computer
aided geometric design \cite{Sederberg}. These curves and surfaces are widely
used for modeling smooth shapes.\textbf{\ }Furthermore,\textbf{\ }the
multivariate extensions of Bernstein operators have been widely studied, with
the most common form being defined on the unit simplex in higher dimensions
\cite{BIE,BSX,CD,Schurer,Stancu}. In the recent paper \cite{RMP}, the authors
present an extension of the Bernstein operator to approximate functions
defined on the unit disk. The shifted $m$th Bernstein--Stancu operator and the
shifted $m$th Bernstein-type operator are defined and their approximation
properties are investigated. Several examples are discussed, comparing the
approximation results of both Bernstein-type operators on the unit disk.

One of the generalization of Bernstein operators is Bernstein-Chlodowsky operators.
These polynomials, introduced by Chlodowsky \cite{Chlodowsky} in 1937, extend the concept of
Bernstein polynomials (from 1912) to an unbounded set. $m$th
Bernstein-Chlodowsky operator is defined as%
\begin{equation}
B_{m}^{Cl}\left(  g,x\right)  =%
{\displaystyle \sum \limits_{k=0}^{m}}
g\left(  \frac{k}{m}b_{m}\right)  q_{m,k}\left(  \frac{x}{b_{m}}\right)  ,\label{2}%
\end{equation}
where%
\begin{equation}
q_{m,k}\left(  \frac{x}{b_{m}}\right)  =\binom{m}{k}\left(  \frac{x}{b_{m}}\right)
^{k}\left(  1-\frac{x}{b_{m}}\right)  ^{m-k}\label{2*}%
\end{equation}
for $0\leq x\leq b_{m},$ $b_{m}\geq0,~\underset{m\rightarrow \infty}{\lim}%
b_{m}=\infty,~\underset{m\rightarrow \infty}{\lim}\frac{b_{m}}{m}=0.$ The
classical Bernstein-Chlodowsky operators are linear and positive operators
acting on the function $g$ and they satisfy%
\begin{equation}
B_{m}^{Cl}\left(  1,x\right)  =1,~B_{m}^{Cl}\left(  t,x\right)  =x\text{ and
}B_{m}^{Cl}\left(  t^{2},x\right)  =x^{2}+\frac{x\left(  b_{m}-x\right)  }%
{m}.\label{1}%
\end{equation}
From the Korovkin's theorem, the operator $B_{m}^{Cl}\left(  g,x\right)  $
converges to $g$ uniformly on the interval $\left[  0,b_{m}\right]  .$

The Bernstein-Chlodowsky operators and their extensions have been studied by
many authors \cite{AOA,GEI,GI,ibikli}. The multivariable generalization of these operators has also been introduced \cite{AA,Buyukyazici, GG,ibikli2,ilarslan}. In \cite{AA}, two dimensional Bernstein-Stancu-Chlodowsky operators on triangle with mobile boundaries are investigated. In \cite{ibikli2}, Bernstein- Chlodowsky polynomials on a triangular domain are studied and the problem of weighted approximations of continuous functions of two variables by a sequences of linear positive operators is discussed.

This paper is motivated by a recent study \cite{RMP} of Bernstein-type
operators on the unit disk. In the present paper, we focus on extending the
Bernstein-Chlodowsky operator for approximating functions  on domains in
$\mathbb{R}^{2}$, such as square region, triangular or elliptic region.  We
explore two modifications: transforming the argument of the target function
and defining a suitable function basis as described in (\ref{2*}). We present
two Bernstein-Chlodowsky type approximants, comparing them through some examples.

The organization of the paper is as follows: in the second section, we present
basic properties of univariate Bernstein-Chlodowsky type operators. In the
third section, we study Bernstein-Chlodowsky type operators in two variables
and, in the next section we discuss some examples of these
operators.\textbf{\ }Finally, we give bivarite shifted $m$th
Bernstein-Chlodowsky-Stancu operators and discuss their approximation
properties.\textbf{\ }

\section{Univariate Bernstein-Chlodowsky type operators}

This section deals with some basic properties of the shifted $m$th Bernstein-Chlodowsky operators which we will use in the next sections. We first start with the shifted Bernstein-Chlodowsky basis on the interval $\left[  \alpha b_{m},\beta b_{m}\right]  $.

From the change of variable%
\[
x=\left(  \beta-\alpha \right)  b_{m}t+\alpha b_{m},~b_{m}\geq0,~\alpha
<\beta,~0\leq t\leq1,
\]
we can define the univariate Bernstein-Chlodowsky basis on the interval
$\left[  \alpha b_{m},\beta b_{m}\right]  $ as follows%
\begin{align}
\widetilde{q}_{m,k}\left(  x;\left[  \alpha b_{m},\beta b_{m}\right]  \right)
& =p_{m,k}\left(  \frac{\frac{x}{b_{m}}-\alpha}{\beta-\alpha}\right)
\nonumber \\
& =\frac{1}{\left(  \beta-\alpha \right)  ^{m}b_{m}^{m}}\binom{m}{k}\left(
x-\alpha b_{m}\right)  ^{k}\left(  \beta b_{m}-x\right)  ^{m-k}\label{6}%
\end{align}
for $\alpha b_{m}\leq x\leq \beta b_{m}.$ For the set of the polynomials
$\widetilde{q}_{m,k}\left(  x;\left[  \alpha b_{m},\beta b_{m}\right]
\right)  $,  the following properties hold:

(i)%
\begin{align}%
{\displaystyle \sum \limits_{k=0}^{m}}
\widetilde{q}_{m,k}\left(  x;\left[  \alpha b_{m},\beta b_{m}\right]  \right)
& =\frac{1}{\left(  \beta-\alpha \right)  ^{m}b_{m}^{m}}%
{\displaystyle \sum \limits_{k=0}^{m}}
\binom{m}{k}\left(  x-\alpha b_{m}\right)  ^{k}\left(  \beta b_{m}-x\right)
^{m-k}\nonumber \\
& =\frac{1}{\left(  \beta-\alpha \right)  ^{m}b_{m}^{m}}\left(  \beta
b_{m}-\alpha b_{m}\right)  ^{m}=1,\label{5}%
\end{align}

(ii) $\widetilde{q}_{m,k}\left(  x;\left[  \alpha b_{m},\beta b_{m}\right]
\right)  \geq0,~\alpha b_{m}\leq x\leq \beta b_{m},$

(iii) $\widetilde{q}_{m,k}\left(  \alpha b_{m}\right)  =p_{m,k}\left(
0\right)  =\delta_{0,k}$ and $\widetilde{q}_{m,k}\left(  \beta b_{m}\right)
=p_{m,k}\left(  1\right)  =\delta_{k,m}$ where $\delta_{k,m}$ denotes the
kronecker delta,

(iv) For $m\neq0,$ the function $\widetilde{q}_{m,k}\left(  x;\left[  \alpha
b_{m},\beta b_{m}\right]  \right)  $ takes maximum value at $x=\left(
\beta-\alpha \right)  \dfrac{kb_{m}}{m}+\alpha b_{m}$ and maximum value equals
to
\[
\widetilde{q}_{m,k}\left(  \left(  \beta-\alpha \right)  \dfrac{kb_{m}}%
{m}+\alpha b_{m};\left[  \alpha b_{m},\beta b_{m}\right]  \right)
=p_{m,k}\left(  \frac{k}{m}\right)  =\binom{m}{k}\frac{k^{k}}{m^{m}}\left(
m-k\right)  ^{m-k}.
\]

(v) $\left(  \beta-\alpha \right)  b_{m}\widetilde{q}_{m,k}^{\prime}\left(
x;\left[  \alpha b_{m},\beta b_{m}\right]  \right)  =m\left(  \widetilde
{q}_{m-1,k-1}\left(  x;\left[  \alpha b_{m},\beta b_{m}\right]  \right)
-\widetilde{q}_{m-1,k}\left(  x;\left[  \alpha b_{m},\beta b_{m}\right]
\right)  \right)  .$
\\
\\
We now consider the shifted univariate $m$th Bernstein-Chlodowsky operator as%
\[
\widetilde{B}_{m}^{Cl}\left(  g\left(  x\right)  ,J\right)  =%
{\displaystyle \sum \limits_{k=0}^{m}}
g\left(  \left(  \beta-\alpha \right)  \dfrac{kb_{m}}{m}+\alpha b_{m}\right)
\widetilde{q}_{m,k}\left(  x,J\right)
\]
for every function $g$ defined on $J=\left[  \alpha b_{m},\beta b_{m}\right]
$. Here $\widetilde{B}_{m}^{Cl}\left(  g\left(  x\right)  ,J\right)  $ is a
polynomial of degree at most $m$ and it follows
\[
\widetilde{B}_{m}^{Cl}\left(  g\left(  x\right)  ,J\right)  =B_{m}%
^{Cl}G\left(  s\right)  ,~~0\leq s\leq b_{m}%
\]
where $G\left(  s\right)  =g\left(  \left(  \beta-\alpha \right)b_{m}  s+\alpha
b_{m}\right)  .$

\begin{lemma}\label{lemma1}
For $x\in \left[ \alpha b_{m},\beta b_{m}\right] $ and $m\in \mathbb{N}$,
the shifted univariate $m$th Bernstein-Chlodowsky operator $\widetilde{B}%
_{m}^{Cl}\left( g\left( x\right) ,J\right) $ has the following results
\end{lemma}

$i)~\widetilde{B}_{m}^{Cl}\left(  1,J\right)  =1,$

$ii)$ $\widetilde{B}_{m}^{Cl}\left(  x,J\right)  =x,$

$iii)~\widetilde{B}_{m}^{Cl}\left(  x^{2},J\right)  =x^{2}+\frac{1}{m}\left(
x-\alpha b_{m}\right)  \left(  \beta b_{m}-x\right)  .$

\begin{proof}
i) It is clear from (\ref{5})%
\[
\widetilde{B}_{m}^{Cl}\left(  1,J\right)  =%
{\displaystyle \sum \limits_{k=0}^{m}}
\widetilde{q}_{m,k}\left(  x,J\right)  =1.
\]

ii) It follows from (\ref{5})
\begin{align*}
\widetilde{B}_{m}^{Cl}\left(  x,J\right)   & =%
{\displaystyle \sum \limits_{k=0}^{m}}
\left(  \left(  \beta-\alpha \right)  \dfrac{kb_{m}}{m}+\alpha b_{m}\right)
\widetilde{q}_{m,k}\left(  x,J\right) \\
& =\left(  \beta-\alpha \right)  b_{m}%
{\displaystyle \sum \limits_{k=0}^{m}}
\frac{k}{m}\widetilde{q}_{m,k}\left(  x,J\right)  +\alpha b_{m}\\
& =\left(  x-\alpha b_{m}\right)
{\displaystyle \sum \limits_{k=0}^{m-1}}
\widetilde{q}_{m-1,k}\left(  x,J\right)  +\alpha b_{m}\\
& =x-\alpha b_{m}+\alpha b_{m}=x
\end{align*}

iii) Finally, we have
\begin{align*}
\widetilde{B}_{m}^{Cl}\left(  x^{2},J\right)   & =%
{\displaystyle \sum \limits_{k=0}^{m}}
\left(  \left(  \beta-\alpha \right)  \dfrac{kb_{m}}{m}+\alpha b_{m}\right)
^{2}\widetilde{q}_{m,k}\left(  x,J\right) \\
& =\left(  \beta-\alpha \right)  ^{2}b_{m}^{2}%
{\displaystyle \sum \limits_{k=0}^{m}}
\frac{k^{2}}{m^{2}}\widetilde{q}_{m,k}\left(  x,J\right)  +2\alpha
b_{m}\left(  x-\alpha b_{m}\right)
{\displaystyle \sum \limits_{k=0}^{m-1}}
\widetilde{q}_{m-1,k}\left(  x,J\right)  +\alpha^{2}b_{m}^{2}\\
& =\left(  \frac{m-1}{m}\right)  \left(  x-\alpha b_{m}\right)  ^{2}%
{\displaystyle \sum \limits_{k=0}^{m-2}}
\widetilde{q}_{m-2,k}\left(  x;J\right)  +\frac{\left(  \beta-\alpha \right)
b_{m}\left(  x-\alpha b_{m}\right)  }{m}%
{\displaystyle \sum \limits_{k=0}^{m-1}}
\widetilde{q}_{m-1,k}\left(  x;J\right) \\
& +2\alpha b_{m}\left(  x-\alpha b_{m}\right)  +\alpha^{2}b_{m}^{2}\\
& =\left(  \frac{m-1}{m}\right)  \left(  x-\alpha b_{m}\right)  ^{2}%
+\frac{\left(  \beta-\alpha \right)  b_{m}\left(  x-\alpha b_{m}\right)  }%
{m}+2\alpha b_{m}\left(  x-\alpha b_{m}\right)  +\alpha^{2}b_{m}^{2}\\
& =x^{2}+\frac{1}{m}\left(  x-\alpha b_{m}\right)  \left(  \beta
b_{m}-x\right)  ,
\end{align*}
which completes the proof.
\end{proof}

By taking into account the results in Lemma \ref{lemma1} , it follows from Korovkin's theorem \cite{AC}:

\begin{theorem}\label{theorem1}
Let $g$ be a continuous function on $\left[ \alpha b_{m},\beta b_{m}\right] $%
. Then%
\begin{equation*}
\underset{m\rightarrow \infty }{\lim }\widetilde{B}_{m}^{Cl}\left(
g,J\right) =g\left( x\right)
\end{equation*}%
holds and the operators $\widetilde{B}_{m}^{Cl}$ converge uniformly on $%
\left[ \alpha b_{m},\beta b_{m}\right] .$
\end{theorem}

\section{Bivariate Bernstein-Chlodowsky-Stancu operators}

D.D. Stancu introduced a method for obtaining polynomials of Bernstein type of two variables in his paper \cite{Stancu}. This method can be applied to define a bivariate operator from univariate Bernstein-Chlodowsky operators $B_{m}^{Cl}\left(  g,x\right)  $.

Assume that $\varphi_{1}=\varphi_{1}\left(  x\right)  $ and $\varphi
_{2}=\varphi_{2}\left(  x\right)  $ are continuous functions satisfying
$\varphi_{1}<\varphi_{2}$ on $\left[  0, b_{m}\right]  .$
Let $\Delta \subset \mathbb{R}^{2}$ be a domain which is bounded by the curves
$y=\varphi_{1}\left(  x\right)  ,~y=\varphi_{2}\left(  x\right)  $ and the
straight lines $x=0$, $x= b_{m}.$ By taking into account
\[
y=\left(  \varphi_{2}\left(  x\right)  -\varphi_{1}\left(  x\right)  \right)
\frac{t}{b_{m}}+\varphi_{1}\left(  x\right)  ,
\]
for every function $g\left(  x,y\right)  $ on the domain $\Delta$, consider%
\[
G\left(  x,t\right)  =g\left(  x,\left(  \varphi_{2}\left(  x\right)
-\varphi_{1}\left(  x\right)  \right)  \frac{t}{b_{m}}+\varphi_{1}\left(
x\right)  \right)  ,~0\leq t\leq b_{m}.
\]
We can define the $m$th Bernstein--Chlodowsky-Stancu operator as follows%
\begin{equation}
\mathcal{B}_{m}^{Cl}\left[  g\left(  x,y\right)  ,\Delta \right]  =%
{\displaystyle \sum \limits_{k=0}^{m}}
{\displaystyle \sum \limits_{j=0}^{m_{k}}}
G\left(  \frac{kb_{m}}{m},\frac{jb_{m}}{m_{k}}\right)  q_{m,k}\left(  \frac
{x}{b_{m}}\right)  q_{m_{k},j}\left(  \frac{t}{b_{m}}\right)  ,\label{3}%
\end{equation}
where the integer $m_{k}$ is nonnegative and corresponds to the $k$th node,
given by $x_{k}=\frac{kb_{m}}{m}.$ In the explicit form, the operator
$\mathcal{B}_{m}^{Cl}$ can be rewritten as%
\[
\mathcal{B}_{m}^{Cl}\left[  g\left(  x,y\right)  ,\Delta \right]  =%
{\displaystyle \sum \limits_{k=0}^{m}}
{\displaystyle \sum \limits_{j=0}^{m_{k}}}
G\left(  \frac{kb_{m}}{m},\frac{jb_{m}}{m_{k}}\right)  q_{m,k}\left(  \frac
{x}{b_{m}}\right)  q_{m_{k},j}\left(  \frac{y-\varphi_{1}\left(  x\right)
}{\varphi_{2}\left(  x\right)  -\varphi_{1}\left(  x\right)  }\right)  ,
\]
from which, it follows%
\[
\mathcal{B}_{m}^{Cl}\left[  g\left(  x,y\right)  ,\Delta \right]  =%
{\displaystyle \sum \limits_{k=0}^{m}}
\left[  B_{m_{k}}^{Cl(t)}G\left(  \frac{kb_{m}}{m},t\right)  \right]
q_{m,k}\left(  \frac{x}{b_{m}}\right)  ,
\]
where $B_{m}^{Cl(t)}$denotes the univariate Bernstein-Chlodowsky operator
acting on the variable $t$.

Notice that the partition step size along the $x$-axis is $\frac{b_{m}}{m}.$
For a fixed node $x_{k}=\frac{k}{m}b_{m}$, the partition step size along the
$t$-axis is $\frac{b_{m}}{m_{k}}$ . Thus, the partition step size along the
$y$-axis is $\frac{b_{m}}{n_{k}}$, where%
\[
n_{k}=\frac{m_{k}}{\varphi_{2}\left(  \frac{k}{m}b_{m}\right)  -\varphi
_{1}\left(  \frac{k}{m}b_{m}\right)  },
\]
from which, we have%
\[
G\left(  \frac{k}{m}b_{m},\frac{j}{m_{k}}b_{m}\right)  =g\left(  \frac{k}{m}b_{m}%
,\frac{j}{n_{k}}+\varphi_{1}\left(  \frac{k}{m}b_{m}\right)  \right)  .
\]

\section{Examples of Bernstein-Chlodowsky-Type Operators Under a Domain
Transformation}

In general, $\mathcal{B}_{m}^{Cl}\left[  g\left(  x,y\right)  ,\Delta \right]
$ is not a polynomial. However, by selecting $\varphi
_{1},~\varphi_{2}$ and $m_{k}$ appropriately, one can obtain polynomials in
the following examples. In this section, under a suitable domain transformation, we extend Bernstein-Chlodowsky-type operators on the square domain to another domain in $\mathbb{R}^{2}$.\\
Under the choices of $\varphi_{1}=0$ and $\varphi_{2}=b_{m},$ we consider the
square $S=[0,b_{m}]\times \lbrack0,b_{m}].$ For a function $g$ defined on $S $,
we have the Bernstein-Chlodowsky-Stancu operator on the square as%
\[
\mathcal{B}_{m}^{Cl}\left[  g\left(  x,y\right)  ,S\right]  =%
{\displaystyle \sum \limits_{k=0}^{m}}
{\displaystyle \sum \limits_{j=0}^{m_{k}}}
g\left(  \frac{kb_{m}}{m},\frac{jb_{m}}{m_{k}}\right)  q_{m,k}\left(  \frac
{x}{b_{m}}\right)  q_{m_{k},j}\left(  \frac{y}{b_{m}}\right)  .
\]

\subsection{The Bernstein-Chlodowsky-Stancu operators on the triangular domain}
The transformation%
\[
\frac{x}{b_{m}}=2u-1\text{ and }\frac{y}{b_{m}}=1-v(1-|2u-1|)
\]
maps the square domain $S=[0,b_{m}]\times \lbrack0,b_{m}]$ into the triangular
domain
\[
S_{b_{m}}=\{(x,y)\in \mathbb{R}^{2}\mid-b_{m}\leq x\leq b_{m},\,y\geq
x,\,y\geq-x,\,y\leq b_{m}\}.
\]
For every
function $g$ defined on $S_{b_{m}}$, we can define the function
$G:S\rightarrow%
\mathbb{R}
^{2}$ in the form
\[
G(u,v)=g((2u-1)b_{m},(1-v(1-|2u-1|))b_{m}),(u,v)\in S.
\]

Under this transformation, the Bernstein-Chlodowsky-Stancu operator on the
triangular domain $S_{b_{m}}$ is in the following form%

\begin{align*}
\widehat{\mathcal{B}}_{m}^{Cl}\left[  g\left(  x,y\right)  ,S_{b_{m}}\right]
& =\sum_{k=0}^{m}\sum_{j=0}^{m_{k}}g\left(  b_{m}(2u-1),b_{m}\left(
1-v(1-|2u-1|)\right)  \right)  q_{m,k}(2u-1)\\
& \times q_{m_{k},j}\left(  1-v(1-|2u-1|)\right) \\
& =\sum_{k=0}^{m}\sum_{j=0}^{m-k}g\left(  b_{m}(2u-1),b_{m}\left(  1-v\left(
1-|2u-1|\right)  \right)  \right) \\
& \times \binom{m}{k}(2u-1)^{k}\left(  1-(2u-1)\right)  ^{m-k}\binom{m-k}{j}\\
& \times \left(  1-v\left(  1-|2u-1|\right)  \right)  ^{j}\left(  v\left(
1-|2u-1|\right)  \right)  ^{m-k-j}%
\end{align*}
where $m_{k}=m-k.$

\begin{figure}[h!]
\centering
\includegraphics[height=2.2528in, width=6.0917in]{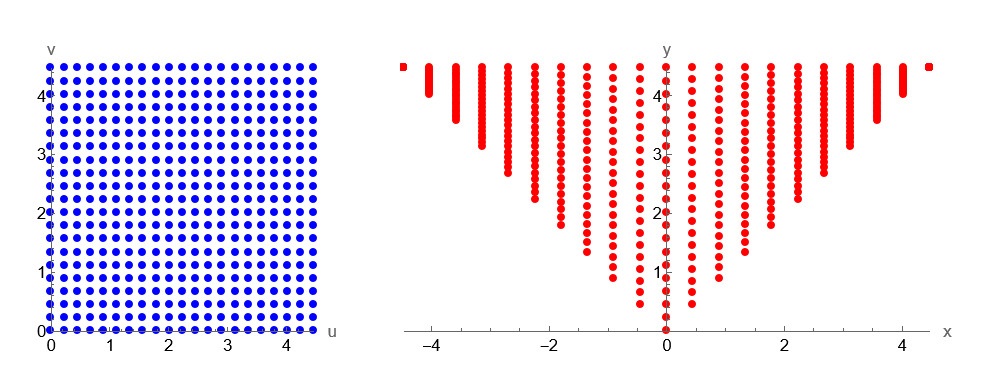}
\caption{}
\end{figure}

\subsection{The Bernstein-Chlodowsky-Stancu operators on the elliptic domain}

In this subsection, we consider two different transformations which map the
square domain into a elliptic domain.

Let%
\[
\mathbb{B}_{b_{m}}=\left \{  \left(  x,y\right)  \in%
\mathbb{R}
^{2}:\frac{x^{2}}{b_{m}^{2}}+\frac{y^{2}}{b_{m}^{2}}\leq1\right \}  .
\]
The transformation $x=(2u-b_{m}),y=(2v-b_{m})\sqrt{1-(2u-b_{m})^{2}}$ maps the
square $S=[0,b_{m}]\times \lbrack0,b_{m}]$ into $\mathbb{B}_{b_{m}}.$ For every
function $g$ defined on $\mathbb{B}_{b_{m}}$, we can define the function
$G:S\rightarrow%
\mathbb{R}
^{2}$ as
\[
G(u,v)=g(2u-b_{m},(2v-b_{m})\sqrt{1-(2u-b_{m})^{2}}),(u,v)\in S.
\]
Under this map, the operator transforms to the following operator%
\begin{align*}
\widehat{\mathcal{B}}_{m}^{Cl}\left[  g\left(  x,y\right)  ,\mathbb{B}_{b_{m}%
}\right]   & =%
{\displaystyle \sum \limits_{k=0}^{m}}
{\displaystyle \sum \limits_{j=0}^{m_{k}}}
g\left(  \frac{2k-m}{m}b_{m},\frac{2j-m_{k}}{m_{k}}\frac{\sqrt{m^{2}-{b_{m}}^2(2k-m)^{2}}}%
{m}b_{m}\right) \\
& \times q_{m,k}\left(  \frac{x+b_{m}}{2b_{m}}\right)  q_{m_{k},j}\left(
\frac{\frac{y}{\sqrt{1-x^{2}}}+b_{m}}{2b_{m}}\right)
\end{align*}
where%
\begin{align*}
q_{m,k}\left(  \frac{x+b_{m}}{2b_{m}}\right)  q_{m_{k},j}\left(  \frac{\frac
{y}{\sqrt{1-x^{2}}}+b_{m}}{2b_{m}}\right)   & =\binom{m}{k}\binom{m_{k}}{j}%
\frac{\left(  x+b_{m}\right)  ^{k}\left(  b_{m}-x\right)  ^{m-k}}{(2b_{m})^{m+m_{k}}\left(  \sqrt{1-x^{2}}\right)  ^{m_{k}}}\\
& \times \left(  y+b_{m}\sqrt{1-x^{2}}\right)  ^{j}\left( b_{m}\sqrt{1-x^{2}%
}-y\right)  ^{m_{k}-j}.
\end{align*}
Since for $y=0$ it follows
\begin{align*}
q_{m,k}\left(  \frac{x+b_{m}}{2b_{m}}\right)  q_{m_{k},j}\left(  \frac{1}
{2}\right)   & =\frac{1}{2^{m+m_{k}}b_{m}^{m}}\binom{m}{k}\binom{m_{k}}{j}\left(
x+b_{m}\right)  ^{k}\left(  b_{m}-x\right)  ^{m-k}%
\end{align*}
and for $x=0$
\begin{align*}
q_{m,k}\left(  \frac{1}{2}\right)  q_{m_{k},j}\left(  \frac{y+b_{m}}%
{2b_{m}}\right)   & =\frac{1}{2^{m+m_{k}}b_{m}^{m_{k}}}\binom{m}{k}\binom{m_{k}}{j} \left(  y+b_{m}\right)  ^{j}\left(  b_{m}-y\right)  ^{m_{k}-j},
\end{align*}
it is seen that $\widehat{\mathcal{B}}_{m}^{Cl}\left[  g\left(  x,y\right)  ,\mathbb{B}_{b_{m}%
}\right]$ is a polynomial for any $m_{k}$ on  the $x$-axis and $y$-axis.

We now consider another map to obtain another bivariate
Bernstein-Chlodowsky-Stancu operator on the elliptic domain $\mathbb{B}_{b_{m}}$.

If we apply the transformation
\[
x=\frac{u}{b_{m}}\sqrt{b_{m}^{2}-v^{2}},\quad y=v
\]
where $(u,v)\in S=[0,b_{m}]\times \lbrack0,b_{m}],$ the square $S$ maps into the
first quadrant $\mathbb{B}_{b_{m}}^{\left(  1\right)  }.$ Similarly, for every
function $g$ defined on $\mathbb{B}_{b_{m}}$, we can define the function
$G:S\rightarrow%
\mathbb{R}
^{2}~$\  \  which maps each quadrant to $S$. The corresponding bivariate
Bernstein-Chlodowsky-Stancu operators on the four quadrants of $\mathbb{B}_{b_{m}}$ are given as follows:%
\begin{align*}
\widehat{\mathcal{B}}_{m}^{Cl}\left[  g\left(  x,y\right)  ,\mathbb{B}_{b_{m}}%
^{\left(  1\right)  }\right]   & =%
{\displaystyle \sum \limits_{k=0}^{m}}
{\displaystyle \sum \limits_{j=0}^{m_{k}}}
g\left(  \frac{kb_{m}\sqrt{m_{k}^{2}-j^{2}}}{m_{k}m},\frac{j}{m_{k}%
}b_{m}\right)  q_{m,k}\left(  \frac{x}{\sqrt{b_{m}^{2}-y^{2}}}\right)
q_{m_{k},j}\left(  \frac{y}{b_{m}}\right) \\
\widehat{\mathcal{B}}_{m}^{Cl}\left[  g\left(  x,y\right)  ,\mathbb{B}_{b_{m}}%
^{\left(  2\right)  }\right]   & =%
{\displaystyle \sum \limits_{k=0}^{m}}
{\displaystyle \sum \limits_{j=0}^{m_{k}}}
g \left( - \frac{kb_{m}\sqrt{m_{k}^{2}-j^{2}}}{m_{k}m},\frac{j}{m_{k}%
}b_{m}\right)  q_{m,k}\left(  \frac{x}{\sqrt{b_{m}^{2}-y^{2}}}\right)
q_{m_{k},j}\left(  \frac{y}{b_{m}}\right) \\
\widehat{\mathcal{B}}_{m}^{Cl}\left[  g\left(  x,y\right)  ,\mathbb{B}_{b_{m}}%
^{\left(  3\right)  }\right]   & =%
{\displaystyle \sum \limits_{k=0}^{m}}
{\displaystyle \sum \limits_{j=0}^{m_{k}}}
g\left( - \frac{kb_{m}\sqrt{m_{k}^{2}-j^{2}}}{m_{k}m},-\frac{j}{m_{k}%
}b_{m}\right) q_{m,k}\left(  \frac{x}{\sqrt{b_{m}^{2}-y^{2}}}\right)q_{m_{k},j}\left(  \frac{y}{b_{m}}\right) \\
\widehat{\mathcal{B}}_{m}^{Cl}\left[  g\left(  x,y\right)  ,\mathbb{B}_{b_{m}}%
^{\left(  4\right)  }\right]   & =%
{\displaystyle \sum \limits_{k=0}^{m}}
{\displaystyle \sum \limits_{j=0}^{m_{k}}}
g\left(  \frac{kb_{m}\sqrt{m_{k}^{2}-j^{2}}}{m_{k}m},-\frac{j}{m_{k}%
}b_{m}\right) q_{m,k}\left(  \frac{x}{\sqrt{b_{m}^{2}-y^{2}}}\right) q_{m_{k},j}\left(  \frac{y}{b_{m}}\right)
\end{align*}
where the functions $G_{i}(u,v)$ ($i=1,2,3,4)$ on $S$ are as follows for every function $g$ on $\mathbb{B}_{b_{m}}$
\begin{align*}
G_{1}(u,v)  & =g(\frac{u}{b_{m}}\sqrt{b_{m}^{2}-v^{2}},v)\\
G_{2}(u,v)  & =g(-\frac{u}{b_{m}}\sqrt{b_{m}^{2}-v^{2}},v)\\
G_{3}(u,v)  & =g(-\frac{u}{b_{m}}\sqrt{b_{m}^{2}-v^{2}},-v)\\
G_{4}(u,v)  & =g(\frac{u}{b_{m}}\sqrt{b_{m}^{2}-v^{2}},-v).
\end{align*}
If we choose $m_{k}=m-k,$ we have%
\[
\widehat{\mathcal{\bar{B}}}_{m}^{Cl}\left[  g\left(  x,y\right)
,\mathbb{B}_{b_{m}}\right]  =\left \{
\begin{array}
[c]{c}%
\widehat{\mathcal{B}}_{m}^{Cl}\left[  g\left(  x,y\right)  ,\mathbb{B}_{b_{m}}%
^{\left(  1\right)  }\right]  ,\left(  x,y\right)  \in \mathbb{B}_{b_{m}}^{\left(
1\right)  }\\
\widehat{\mathcal{B}}_{m}^{Cl}\left[  g\left(  x,y\right)  ,\mathbb{B}_{b_{m}}%
^{\left(  2\right)  }\right]  ,\left(  x,y\right)  \in \mathbb{B}_{b_{m}}^{\left(
2\right)  }\\
\widehat{\mathcal{B}}_{m}^{Cl}\left[  g\left(  x,y\right)  ,\mathbb{B}_{b_{m}}%
^{\left(  3\right)  }\right]  ,\left(  x,y\right)  \in \mathbb{B}_{b_{m}}^{\left(
3\right)  }\\
\widehat{\mathcal{B}}_{m}^{Cl}\left[  g\left(  x,y\right)  ,\mathbb{B}_{b_{m}}%
^{\left(  4\right)  }\right]  ,\left(  x,y\right)  \in \mathbb{B}_{b_{m}}^{\left(
4\right)  }%
\end{array}
\right.
\]%

\begin{figure}
[ptbh]
\begin{center}
\includegraphics[
height=3.339in,
width=3.2102in
]%
{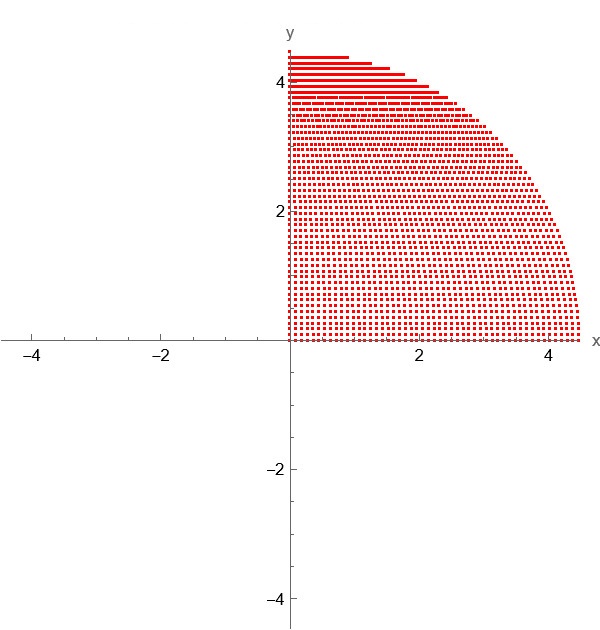}%
\caption{$b_{m}=\sqrt{m},~m=20$}%
\end{center}
\end{figure}

\section{Bivariate Shifted Bernstein-Chlodowsky-Stancu operators}

In this section, we introduce bivariate shifted $m$th
Bernstein-Chlodowsky-Stancu operators and discuss some their approximation properties.

Assume that $\varphi_{1}\left(  x\right)  $ and $\varphi_{2}\left(  x\right)
$ are continuous functions satisfying $\varphi_{1}<\varphi_{2}$ on the
interval $J=\left[  \alpha b_{m},\beta b_{m}\right]  .$ Let $\Delta
\subset \mathbb{R}^{2}$ be a domain which is bounded by the curves
$y=\varphi_{1}\left(  x\right)  ,~y=\varphi_{2}\left(  x\right)  $ and the
straight lines $x=\alpha b_{m}$, $x=\beta b_{m}.$ For a fixed $x\in J,$ the
polynomial $\widetilde{q}_{m,k}\left(  y;\left[  \varphi_{1}\left(  x\right)
,\varphi_{2}\left(  x\right)  \right]  \right)  ,~0\leq k\leq m,~m\geq0$
denotes a univariate shifted Bernstein-Chlodowsky basis on the interval
$\left[  \varphi_{1}\left(  x\right)  ,\varphi_{2}\left(  x\right)  \right]
.$

For every function $g\left(  x,y\right)  $ on $\Delta,$ we consider the
function%
\[
\widetilde{G}\left(  u,v;\Delta \right)  =g\left(  \left(  \beta b_{m}-\alpha
b_{m}\right)  u+\alpha b_{m},\left(  \widetilde{\varphi}_{2}\left(  u\right)
-\widetilde{\varphi}_{1}\left(  u\right)  \right)  v+\widetilde{\varphi}%
_{1}\left(  u\right)  \right)
\]
where $0\leq u,v\leq1$ and%
\[
\widetilde{\varphi}_{i}\left(  u\right)  =\varphi_{i}\left(  \left(  \beta
b_{m}-\alpha b_{m}\right)  u+\alpha b_{m}\right)  ;~i=1,2.
\]

The shifted $m$th Bernstein-Chlodowsky-Stancu operator is as follows%
\[
\widetilde{\mathcal{B}}_{m}^{Cl}\left[  g\left(  x,y\right)  ;\Delta \right]  =%
{\displaystyle \sum \limits_{k=0}^{m}}
{\displaystyle \sum \limits_{j=0}^{m_{k}}}
\widetilde{G}\left(  \frac{k}{m},\frac{j}{m_{k}};\Delta \right)  \widetilde
{q}_{m,k}\left(  x;J\right)  \widetilde{q}_{m_{k},j}\left(  y;\left[
\varphi_{1},\varphi_{2}\right]  \right)  ,
\]
where%
\[
\widetilde{q}_{m,k}\left(  x;J\right)  =p_{m,k}\left(  \frac{x-\alpha b_{m}%
}{\left(  \beta-\alpha \right)  b_{m}}\right)  ~\text{and }\widetilde{q}%
_{m,k}\left(  y;\left[  \varphi_{1},\varphi_{2}\right]  \right)
=p_{m,k}\left(  \frac{y-\varphi_{1}\left(  x\right)  }{\varphi_{2}\left(
x\right)  -\varphi_{1}\left(  x\right)  }\right)
\]
and $m_{k}=m-k$ or $m_{k}=k$ for $0\leq k\leq m$.
In the explicit form, the univariate Bernstein-Chlodowsky is defined as%
\begin{align*}
\widetilde{\mathcal{B}}_{m}^{Cl}\left[  g\left(  x,y\right)  ;\Delta \right]   & =%
{\displaystyle \sum \limits_{k=0}^{m}}
{\displaystyle \sum \limits_{j=0}^{m_{k}}}
\widetilde{G}\left(  \frac{k}{m},\frac{j}{m_{k}};\Delta \right)  p_{m,k}\left(
\frac{x-\alpha b_{m}}{\left(  \beta-\alpha \right)  b_{m}}\right) \\
& \times p_{m_{k},j}\left(  \frac{y-\varphi_{1}\left(  x\right)  }{\varphi
_{2}\left(  x\right)  -\varphi_{1}\left(  x\right)  }\right)  .
\end{align*}
For the operator $\widetilde{\mathcal{B}}_{m}^{Cl}\left[  g\left(  x,y\right)
;\Delta \right]  ,$ the following results hold true.

\begin{lemma}\label{lemma2}
Assume that $\varphi _{1}\left( x\right) $ and $\varphi _{2}\left( x\right) $
are continuous functions satisfying $\varphi _{1}<\varphi _{2}$ on the
interval $J=\left[ \alpha b_{m},\beta b_{m}\right] .$ Let $\Delta \subset
\mathbb{R}^{2}$ be a domain which is bounded by the curves $y=\varphi
_{1}\left( x\right) ,~y=\varphi _{2}\left( x\right) $ and the straight lines
$x=\alpha b_{m}$, $x=\beta b_{m}.$ Then, the following results are satisfied\\
\\
(i) $\widetilde{\mathcal{B}}_{m}^{Cl}\left[ 1;\Delta \right] =1$\\
(ii) $\widetilde{\mathcal{B}}_{m}^{Cl}\left[ x;\Delta \right] =x$\\
(iii) $\widetilde{\mathcal{B}}_{m}^{Cl}\left[ y;\Delta \right] \rightarrow y$ uniformly on the interval $ J $ when $ m\rightarrow \infty$ \\
(iv) $\widetilde{\mathcal{B}}_{m}^{Cl}\left[ x^{2};\Delta \right] =x^{2}+\frac{1}{m}%
\left( x-\alpha b_{m}\right) \left( \beta b_{m}-x\right) $\\
(v) $\widetilde{\mathcal{B}}_{m}^{Cl}\left[ y^{2};\Delta \right] \rightarrow y^{2}$ uniformly on the interval $J$ when $m\rightarrow \infty $.
\end{lemma}

\begin{proof}
(i) It follows from (\ref{6}) and (\ref{5})
\begin{align*}
\widetilde{\mathcal{B}}_{m}^{Cl}\left[  1;\Delta \right]   & =%
{\displaystyle \sum \limits_{k=0}^{m}}
{\displaystyle \sum \limits_{j=0}^{m_{k}}}
p_{m,k}\left(  \frac{x-\alpha b_{m}}{\left(  \beta-\alpha \right)  b_{m}%
}\right)  p_{m_{k},j}\left(  \frac{y-\varphi_{1}\left(  x\right)  }%
{\varphi_{2}\left(  x\right)  -\varphi_{1}\left(  x\right)  }\right) \\
& =%
{\displaystyle \sum \limits_{k=0}^{m}}
{\displaystyle \sum \limits_{j=0}^{m_{k}}}
\widetilde{q}_{m,k}\left(  x;J\right)  \widetilde{q}_{m_{k},j}\left(
y;\left[  \varphi_{1},\varphi_{2}\right]  \right)  =%
{\displaystyle \sum \limits_{j=0}^{m_{k}}}
\widetilde{q}_{m_{k},j}\left(  y;\left[  \varphi_{1},\varphi_{2}\right]
\right) \\
& =%
{\displaystyle \sum \limits_{j=0}^{m_{k}}}
\binom{m_{k}}{j}\left(  \frac{y-\varphi_{1}\left(  x\right)  }{\varphi
_{2}\left(  x\right)  -\varphi_{1}\left(  x\right)  }\right)  ^{j}\left(
1-\frac{y-\varphi_{1}\left(  x\right)  }{\varphi_{2}\left(  x\right)
-\varphi_{1}\left(  x\right)  }\right)  ^{m_{k}-j}\\
& =%
{\displaystyle \sum \limits_{j=0}^{m_{k}}}
\binom{m_{k}}{j}\left(  y-\varphi_{1}\left(  x\right)  \right)  ^{j}\left(
\varphi_{2}\left(  x\right)  -y\right)  ^{m_{k}-j}\frac{1}{\left(  \varphi
_{2}\left(  x\right)  -\varphi_{1}\left(  x\right)  \right)  ^{m_{k}}}\\
& =\left(  \varphi_{2}\left(  x\right)  -\varphi_{1}\left(  x\right)  \right)
^{m_{k}}\frac{1}{\left(  \varphi_{2}\left(  x\right)  -\varphi_{1}\left(
x\right)  \right)  ^{m_{k}}}=1.
\end{align*}

(ii) Since $%
{\displaystyle \sum \limits_{k=0}^{m}}
\widetilde{q}_{m,k}\left(  x;J\right)  =1,$ we have%
\begin{align*}
\widetilde{\mathcal{B}}_{m}^{Cl}\left[  x;\Delta \right]   & =%
{\displaystyle \sum \limits_{k=0}^{m}}
{\displaystyle \sum \limits_{j=0}^{m_{k}}}
\left(  \left(  \beta-\alpha \right)  b_{m}\frac{k}{m}+\alpha b_{m}\right)
\widetilde{q}_{m,k}\left(  x;J\right)  \widetilde{q}_{m_{k},j}\left(
y;\left[  \varphi_{1},\varphi_{2}\right]  \right) \\
& =%
{\displaystyle \sum \limits_{k=0}^{m}}
\left(  \left(  \beta-\alpha \right)  b_{m}\frac{k}{m}+\alpha b_{m}\right)
\widetilde{q}_{m,k}\left(  x;J\right)
{\displaystyle \sum \limits_{j=0}^{m_{k}}}
\widetilde{q}_{m_{k},j}\left(  y;\left[  \varphi_{1},\varphi_{2}\right]
\right) \\
& =\left(  \beta-\alpha \right)  b_{m}%
{\displaystyle \sum \limits_{k=0}^{m}}
\frac{k}{m}\widetilde{q}_{m,k}\left(  x;J\right)  +\alpha b_{m}%
{\displaystyle \sum \limits_{k=0}^{m}}
\widetilde{q}_{m,k}\left(  x;J\right) \\
& =\left(  \beta-\alpha \right)  b_{m}%
{\displaystyle \sum \limits_{k=0}^{m}}
\frac{k}{m}\widetilde{q}_{m,k}\left(  x;J\right)  +\alpha b_{m}\\
& =\left(  \beta-\alpha \right)  b_{m}%
{\displaystyle \sum \limits_{k=0}^{m}}
\frac{k}{m}\frac{1}{\left(  \beta-\alpha \right)  ^{m}b_{m}^{m}}\binom{m}%
{k}\left(  x-\alpha b_{m}\right)  ^{k}\left(  \beta b_{m}-x\right)
^{m-k}+\alpha b_{m}\\
& =\left(  x-\alpha b_{m}\right)
{\displaystyle \sum \limits_{k=0}^{m-1}}
\widetilde{q}_{m-1,k}\left(  x;J\right)  +\alpha b_{m}\\
& =x-\alpha b_{m}+\alpha b_{m}=x
\end{align*}

(iii) We obtain%
\begin{align}
\widetilde{\mathcal{B}}_{m}^{Cl}\left[  y;\Delta \right]   & =%
{\displaystyle \sum \limits_{k=0}^{m}}
{\displaystyle \sum \limits_{j=0}^{m_{k}}}
\left(  \left(  \widetilde{\varphi}_{2}\left(  \frac{k}{m}\right)
-\widetilde{\varphi}_{1}\left(  \frac{k}{m}\right)  \right)  \frac{j}{m_{k}%
}+\widetilde{\varphi}_{1}\left(  \frac{k}{m}\right)  \right) \nonumber \\
& \times \widetilde{q}_{m,k}\left(  x;J\right)  \widetilde{q}_{m_{k},j}\left(
y;\left[  \varphi_{1},\varphi_{2}\right]  \right) \nonumber \\
& =%
{\displaystyle \sum \limits_{k=0}^{m}}
\left(  \widetilde{\varphi}_{2}\left(  \frac{k}{m}\right)  -\widetilde
{\varphi}_{1}\left(  \frac{k}{m}\right)  \right)  \widetilde{q}_{m,k}\left(
x;J\right)
{\displaystyle \sum \limits_{j=0}^{m_{k}}}
\widetilde{q}_{m_{k},j}\left(  y;\left[  \varphi_{1},\varphi_{2}\right]
\right)  \frac{j}{m_{k}}\nonumber \\
& +%
{\displaystyle \sum \limits_{k=0}^{m}}
\widetilde{\varphi}_{1}\left(  \frac{k}{m}\right)  \widetilde{q}_{m,k}\left(
x;J\right)  .\label{7}%
\end{align}
Since $%
{\displaystyle \sum \limits_{j=0}^{m_{k}}}
\widetilde{q}_{m_{k},j}\left(  y;\left[  \varphi_{1},\varphi_{2}\right]
\right)  \frac{j}{m_{k}}=\left(  \frac{y-\varphi_{1}\left(  x\right)
}{\varphi_{2}\left(  x\right)  -\varphi_{1}\left(  x\right)  }\right)
{\displaystyle \sum \limits_{j=0}^{m_{k}-1}}
\widetilde{q}_{m_{k}-1,j}\left(  y;\left[  \varphi_{1},\varphi_{2}\right]
\right)  =\left(  \frac{y-\varphi_{1}\left(  x\right)  }{\varphi_{2}\left(
x\right)  -\varphi_{1}\left(  x\right)  }\right)  ,$ the equation (\ref{7})
reduces to%
\begin{align*}
\widetilde{\mathcal{B}}_{m}^{Cl}\left[  y;\Delta \right]   & =\left(  \frac{y-\varphi
_{1}\left(  x\right)  }{\varphi_{2}\left(  x\right)  -\varphi_{1}\left(
x\right)  }\right)
{\displaystyle \sum \limits_{k=0}^{m}}
\left(  \widetilde{\varphi}_{2}\left(  \frac{k}{m}\right)  -\widetilde
{\varphi}_{1}\left(  \frac{k}{m}\right)  \right)  \widetilde{q}_{m,k}\left(
x;J\right) \\
& +%
{\displaystyle \sum \limits_{k=0}^{m}}
\widetilde{\varphi}_{1}\left(  \frac{k}{m}\right)  \widetilde{q}_{m,k}\left(
x;J\right) \\
& =\left(  \frac{y-\varphi_{1}\left(  x\right)  }{\varphi_{2}\left(  x\right)
-\varphi_{1}\left(  x\right)  }\right)  \widetilde{B}_{m}^{Cl}\left[
\varphi_{2}-\varphi_{1};J\right]  +\widetilde{B}_{m}^{Cl}\left[  \varphi
_{1};J\right]
\end{align*}
where $\widetilde{B}_{m}^{Cl}$ is the univariate shifted Bernstein-Chlodowsky
operators. From theorem \ref{theorem1}, since%
\[
\widetilde{B}_{m}^{Cl}\left[  \varphi_{2}-\varphi_{1};J\right]  \rightarrow
\varphi_{2}-\varphi_{1}\text{ and }\widetilde{B}_{m}^{Cl}\left[  \varphi
_{1};J\right]  \rightarrow \varphi_{1}\text{,}%
\]
we arrive at%
\[
\underset{m\rightarrow \infty}{\lim}\widetilde{\mathcal{B}}_{m}^{Cl}\left[
y;\Delta \right]  =\left(  \frac{y-\varphi_{1}\left(  x\right)  }{\varphi
_{2}\left(  x\right)  -\varphi_{1}\left(  x\right)  }\right)  \left(
\varphi_{2}\left(  x\right)  -\varphi_{1}\left(  x\right)  \right)
+\varphi_{1}\left(  x\right)  =y.
\]

(iv)
\begin{align*}
\widetilde{\mathcal{B}}_{m}^{Cl}\left[  x^{2};\Delta \right]   & =%
{\displaystyle \sum \limits_{k=0}^{m}}
{\displaystyle \sum \limits_{j=0}^{m_{k}}}
\left(  \left(  \beta-\alpha \right)  b_{m}\frac{k}{m}+\alpha b_{m}\right)
^{2}\widetilde{q}_{m,k}\left(  x;J\right)  \widetilde{q}_{m_{k},j}\left(
y;\left[  \varphi_{1},\varphi_{2}\right]  \right) \\
& =\left(  \beta-\alpha \right)  ^{2}b_{m}^{2}%
{\displaystyle \sum \limits_{k=0}^{m}}
\left(  \frac{k}{m}\right)  ^{2}\widetilde{q}_{m,k}\left(  x;J\right)
+2\left(  \beta-\alpha \right)  \alpha b_{m}^{2}%
{\displaystyle \sum \limits_{k=0}^{m}}
\frac{k}{m}\widetilde{q}_{m,k}\left(  x;J\right)  +\alpha^{2}b_{m}^{2}\\
& =\left(  \frac{m-1}{m}\right)  \left(  x-\alpha b_{m}\right)  ^{2}%
{\displaystyle \sum \limits_{k=0}^{m-2}}
\widetilde{q}_{m-2,k}\left(  x;J\right)  +\frac{\left(  \beta-\alpha \right)
b_{m}\left(  x-\alpha b_{m}\right)  }{m}%
{\displaystyle \sum \limits_{k=0}^{m-1}}
\widetilde{q}_{m-1,k}\left(  x;J\right) \\
& +2\alpha b_{m}\left(  x-\alpha b_{m}\right)
{\displaystyle \sum \limits_{k=0}^{m-1}}
\widetilde{q}_{m-1,k}\left(  x;J\right)  +\alpha^{2}b_{m}^{2}\\
& =\left(  \frac{m-1}{m}\right)  \left(  x-\alpha b_{m}\right)  ^{2}%
+\frac{\left(  \beta-\alpha \right)  b_{m}\left(  x-\alpha b_{m}\right)  }%
{m}+2\alpha b_{m}\left(  x-\alpha b_{m}\right)  +\alpha^{2}b_{m}^{2}\\
& =x^{2}+\frac{1}{m}\left(  x-\alpha b_{m}\right)  \left(  \beta
b_{m}-x\right)  .
\end{align*}
(v) Finally, it follows%
\begin{align*}
\widetilde{\mathcal{B}}_{m}^{Cl}\left[  y^{2};\Delta \right]   & =%
{\displaystyle \sum \limits_{k=0}^{m}}
{\displaystyle \sum \limits_{j=0}^{m_{k}}}
\left(  \left(  \widetilde{\varphi}_{2}\left(  \frac{k}{m}\right)
-\widetilde{\varphi}_{1}\left(  \frac{k}{m}\right)  \right)  \frac{j}{m_{k}%
}+\widetilde{\varphi}_{1}\left(  \frac{k}{m}\right)  \right)  ^{2}\\
& \times \widetilde{q}_{m,k}\left(  x;J\right)  \widetilde{q}_{m_{k},j}\left(
y;\left[  \varphi_{1},\varphi_{2}\right]  \right) \\
& =%
{\displaystyle \sum \limits_{k=0}^{m}}
\left(  \widetilde{\varphi}_{2}\left(  \frac{k}{m}\right)  -\widetilde
{\varphi}_{1}\left(  \frac{k}{m}\right)  \right)  ^{2}\widetilde{q}%
_{m,k}\left(  x;J\right)
{\displaystyle \sum \limits_{j=0}^{m_{k}}}
\widetilde{q}_{m_{k},j}\left(  y;\left[  \varphi_{1},\varphi_{2}\right]
\right)  \left(  \frac{j}{m_{k}}\right)  ^{2}\\
& +2%
{\displaystyle \sum \limits_{k=0}^{m}}
{\displaystyle \sum \limits_{j=0}^{m_{k}}}
\left(  \widetilde{\varphi}_{2}\left(  \frac{k}{m}\right)  -\widetilde
{\varphi}_{1}\left(  \frac{k}{m}\right)  \right)  \left(  \frac{j}{m_{k}%
}\right)  \widetilde{\varphi}_{1}\left(  \frac{k}{m}\right)  \widetilde
{q}_{m,k}\left(  x;J\right)  \widetilde{q}_{m_{k},j}\left(  y;\left[
\varphi_{1},\varphi_{2}\right]  \right) \\
& +%
{\displaystyle \sum \limits_{k=0}^{m}}
{\displaystyle \sum \limits_{j=0}^{m_{k}}}
\left(  \widetilde{\varphi}_{1}\left(  \frac{k}{m}\right)  \right)
^{2}\widetilde{q}_{m,k}\left(  x;J\right)  \widetilde{q}_{m_{k},j}\left(
y;\left[  \varphi_{1},\varphi_{2}\right]  \right) \\
& =\widetilde{B}_{m}^{Cl}\left[  \varphi_{2}^{2}-\varphi_{1}^{2};J\right]
\left(  \frac{y-\varphi_{1}\left(  x\right)  }{\varphi_{2}\left(  x\right)
-\varphi_{1}\left(  x\right)  }\right)  ^{2}\\
& +\frac{\left(  y-\varphi_{1}\left(  x\right)  \right)  \left(  \varphi
_{2}\left(  x\right)  -y\right)  }{\left(  \varphi_{2}\left(  x\right)
-\varphi_{1}\left(  x\right)  \right)  ^{2}}%
{\displaystyle \sum \limits_{k=0}^{m}}
\frac{1}{m_{k}}\left(  \widetilde{\varphi}_{2}\left(  \frac{k}{m}\right)
-\widetilde{\varphi}_{1}\left(  \frac{k}{m}\right)  \right)  ^{2}\widetilde
{q}_{m,k}\left(  x;J\right) \\
& +2\frac{y-\varphi_{1}\left(  x\right)  }{\varphi_{2}\left(  x\right)
-\varphi_{1}\left(  x\right)  }%
{\displaystyle \sum \limits_{k=0}^{m}}
\left(  \widetilde{\varphi}_{2}\left(  \frac{k}{m}\right)  -\widetilde
{\varphi}_{1}\left(  \frac{k}{m}\right)  \right)  \widetilde{\varphi}%
_{1}\left(  \frac{k}{m}\right)  \widetilde{q}_{m,k}\left(  x;J\right)
+\widetilde{B}_{m}^{Cl}\left[  \varphi_{1}^{2};J\right] \\
& =\widetilde{B}_{m}^{Cl}\left[  \varphi_{2}^{2}-\varphi_{1}^{2};J\right]
\left(  \frac{y-\varphi_{1}\left(  x\right)  }{\varphi_{2}\left(  x\right)
-\varphi_{1}\left(  x\right)  }\right)  ^{2}+2\frac{y-\varphi_{1}\left(
x\right)  }{\varphi_{2}\left(  x\right)  -\varphi_{1}\left(  x\right)
}\widetilde{B}_{m}^{Cl}\left[  \left(  \varphi_{2}-\varphi_{1}\right)
\varphi_{1};J\right] \\
& +\widetilde{B}_{m}^{Cl}\left[  \varphi_{1}^{2};J\right]  +Q_{m}\left(
x\right)
\end{align*}
where%
\[
Q_{m}\left(  x\right)  =\frac{\left(  y-\varphi_{1}\left(  x\right)  \right)
\left(  \varphi_{2}\left(  x\right)  -y\right)  }{\left(  \varphi_{2}\left(
x\right)  -\varphi_{1}\left(  x\right)  \right)  ^{2}}%
{\displaystyle \sum \limits_{k=0}^{m}}
\frac{1}{m_{k}}\left(  \widetilde{\varphi}_{2}\left(  \frac{k}{m}\right)
-\widetilde{\varphi}_{1}\left(  \frac{k}{m}\right)  \right)  ^{2}\widetilde
{q}_{m,k}\left(  x;J\right)  .
\]
When we choose $m_{k}=m-kb_{m},$ we have%
\begin{align*}
Q_{m}\left(  x\right)   & =\frac{\left(  y-\varphi_{1}\left(  x\right)
\right)  \left(  \varphi_{2}\left(  x\right)  -y\right)  }{m\left(
\varphi_{2}\left(  x\right)  -\varphi_{1}\left(  x\right)  \right)  ^{2}}%
{\displaystyle \sum \limits_{k=0}^{m}}
\frac{1}{1-kb_{m}/m}\left(  \widetilde{\varphi}_{2}\left(  \frac{k}{m}\right)
-\widetilde{\varphi}_{1}\left(  \frac{k}{m}\right)  \right)  ^{2}\widetilde
{q}_{m,k}\left(  x;J\right) \\
& =\frac{\left(  y-\varphi_{1}\left(  x\right)  \right)  \left(  \varphi
_{2}\left(  x\right)  -y\right)  }{m\left(  \varphi_{2}\left(  x\right)
-\varphi_{1}\left(  x\right)  \right)  ^{2}}\widetilde{B}_{m}^{Cl}\left[
\frac{\varphi_{2}^{2}-\varphi_{1}^{2}}{1-\frac{x-\alpha b_{m}}{\left(
\beta-\alpha \right)  b_{m}}};J\right]
\end{align*}
and in the case of $m_{k}=kb_{m},$ then%
\[
Q_{m}\left(  x\right)  =\frac{\left(  y-\varphi_{1}\left(  x\right)  \right)
\left(  \varphi_{2}\left(  x\right)  -y\right)  }{m\left(  \varphi_{2}\left(
x\right)  -\varphi_{1}\left(  x\right)  \right)  ^{2}}\widetilde{B}_{m}%
^{Cl}\left[  \frac{\varphi_{2}^{2}-\varphi_{1}^{2}}{\frac{x-\alpha b_{m}%
}{\left(  \beta-\alpha \right)  b_{m}}};J\right]  .
\]
Since%
\[
\widetilde{B}_{m}^{Cl}\left[  \varphi_{2}^{2}-\varphi_{1}^{2};J\right]
\rightarrow \varphi_{2}^{2}-\varphi_{1}^{2},~\widetilde{B}_{m}^{Cl}\left[
\left(  \varphi_{2}-\varphi_{1}\right)  \varphi_{1};J\right]  \rightarrow
\left(  \varphi_{2}-\varphi_{1}\right)  \varphi_{1},\text{ and }\widetilde
{B}_{m}^{Cl}\left[  \varphi_{1}^{2};J\right]  \rightarrow \varphi_{1}%
^{2}\text{,}%
\]
from Theorem \ref{theorem1} and $Q_{m}\left(  x\right)  \rightarrow0$ as $m\rightarrow
\infty,$ it follows $\underset{m\rightarrow \infty}{\lim}\widetilde{\mathcal{B}}_{m}^{Cl}\left[  y^{2};\Delta \right]  =y^{2}.$
\end{proof}

In order to give approximation properties of the operator $\widetilde{\mathcal{B}}_{m}^{Cl}\left[  g\left(  x,y\right)  ;\Delta \right]  ,$ we first give the
modulus continuity of the function $g.$

\begin{definition}
\cite{Schurer} For a continuous function $g$ on the domain $\Delta ,$ if $%
\delta _{1}$ and $\delta _{2}$ are positive real numbers, the modulus of
continuity of $g$ is given by%
\begin{equation*}
w\left( \delta _{1},\delta _{2}\right) =\sup \left \vert g\left(
x_{2},y_{2}\right) -g\left( x_{1},y_{1}\right) \right \vert
\end{equation*}%
where the points $\left( x_{1},y_{1}\right) $ and $\left( x_{2},y_{2}\right)
$ are inside $\Delta $ such that $\left \vert x_{2}-x_{1}\right \vert \leq
\delta _{1}$ and $\left \vert y_{2}-y_{1}\right \vert \leq \delta _{2}.$ It is
also well known that, for any $\delta _{1},\delta _{2}>0$ on $\Delta $
\begin{equation}
\left \vert g\left( x_{2},y_{2}\right) -g\left( x_{1},y_{1}\right)
\right \vert \leq w\left( \left \vert x_{2}-x_{1}\right \vert ,\left \vert
y_{2}-y_{1}\right \vert \right) \leq w\left( \delta _{1},\delta _{2}\right) .
\label{9}
\end{equation}%
Also, the inequality
\begin{equation}
w\left( a\delta _{1},b\delta _{2}\right) \leq \left( a+b+1\right) w\left(
\delta _{1},\delta _{2}\right)   \label{8}
\end{equation}%
is satisfied for $a,b>0$ \cite{Schurer,Stancu}.
\end{definition}

\begin{theorem}
For a continuous function $g$ on the domain $\Delta ,$ we have%
\begin{equation*}
\widetilde{\mathcal{B}}_{m}^{Cl}\left[ g\left( x,y\right) ;\Delta \right] \rightarrow
g\left( x,y\right)
\end{equation*}%
as $m\rightarrow \infty $ uniformly on $\Delta .$
\end{theorem}

\begin{proof}
From (\ref{9}) and (\ref{8}), we can write%
\begin{align}
& \left \vert g\left(  x,y\right)  -\widetilde{G}\left(  \frac{k}{m},\frac
{j}{m_{k}};\Delta \right)  \right \vert \nonumber \\
& \leq w\left(  \left \vert x-\left(  \beta-\alpha \right)  b_{m}\frac{k}%
{m}-\alpha b_{m}\right \vert ,\left \vert y-\left(  \widetilde{\varphi}%
_{2}\left(  \frac{k}{m}\right)  -\widetilde{\varphi}_{1}\left(  \frac{k}%
{m}\right)  \right)  \frac{j}{m_{k}}-\widetilde{\varphi}_{1}\left(  \frac
{k}{m}\right)  \right \vert \right) \nonumber \\
& \leq \left(  \theta_{1}+\theta_{2}+1\right)  w\left(  \delta_{1},\delta
_{2}\right)  ,\label{10}%
\end{align}
where%
\[
\theta_{1}=\frac{1}{\delta_{1}}\left \vert x-\left(  \beta-\alpha \right)
b_{m}\frac{k}{m}-\alpha b_{m}\right \vert \text{ and }\theta_{2}=\frac
{1}{\delta_{2}}\left \vert y-\left(  \widetilde{\varphi}_{2}\left(  \frac{k}%
{m}\right)  -\widetilde{\varphi}_{1}\left(  \frac{k}{m}\right)  \right)
\frac{j}{m_{k}}-\widetilde{\varphi}_{1}\left(  \frac{k}{m}\right)  \right \vert
.
\]

By using the fact that $\widetilde{\mathcal{B}}_{m}^{Cl}\left[  1;\Delta \right]  =1$ and
$\widetilde{q}_{m,k}\left(  x;J\right)  \widetilde{q}_{m_{k},j}\left(
y;\left[  \varphi_{1},\varphi_{2}\right]  \right)  \geq0,$ it follows from
(\ref{10})%
\begin{align*}
& \left \vert g\left(  x,y\right)  -\widetilde{\mathcal{B}}_{m}^{Cl}\left[  g\left(
x,y\right)  ;\Delta \right]  \right \vert \\
& \leq%
{\displaystyle \sum \limits_{k=0}^{m}}
{\displaystyle \sum \limits_{j=0}^{m_{k}}}
\widetilde{q}_{m,k}\left(  x;J\right)  \widetilde{q}_{m_{k},j}\left(
y;\left[  \varphi_{1},\varphi_{2}\right]  \right)  \left \vert g\left(
x,y\right)  -\widetilde{G}\left(  \frac{k}{m},\frac{j}{m_{k}};\Delta \right)
\right \vert \\
& \leq%
{\displaystyle \sum \limits_{k=0}^{m}}
{\displaystyle \sum \limits_{j=0}^{m_{k}}}
\widetilde{q}_{m,k}\left(  x;J\right)  \widetilde{q}_{m_{k},j}\left(
y;\left[  \varphi_{1},\varphi_{2}\right]  \right)  \left(  \theta_{1}%
+\theta_{2}+1\right)  w\left(  \delta_{1},\delta_{2}\right) \\
& =I^{(1)}\left(  x,y,k,j,m,m_{k},\delta_{1},\delta_{2}\right)  +I^{(2)}%
\left(  x,y,k,j,m,m_{k},\delta_{1},\delta_{2}\right)  +I^{(3)}\left(
x,y,k,j,m,m_{k},\delta_{1},\delta_{2}\right)  ,
\end{align*}
where%
\[
I^{(i)}\left(  x,y,k,j,m,m_{k},\delta_{1},\delta_{2}\right)  =%
{\displaystyle \sum \limits_{k=0}^{m}}
{\displaystyle \sum \limits_{j=0}^{m_{k}}}
\widetilde{q}_{m,k}\left(  x;J\right)  \widetilde{q}_{m_{k},j}\left(
y;\left[  \varphi_{1},\varphi_{2}\right]  \right)  \theta_{i}w\left(
\delta_{1},\delta_{2}\right)  ,
\]

for $i=0,1,2$ and $\theta_{0}=1$. We now compute each $I^{(i)}\left(
x,y,k,j,m,m_{k},\delta_{1},\delta_{2}\right)  $ for $i=0,1,2.$

For the first term $I^{(1)}\left(  x,y,k,j,m,m_{k},\delta_{1},\delta
_{2}\right)  ,$ since $g\left(  x\right)  =\sqrt{x}$ is a concav function, we
can write from Jensen's inequality \cite{Jensen}%
\begin{align*}
&
{\displaystyle \sum \limits_{k=0}^{m}}
{\displaystyle \sum \limits_{j=0}^{m_{k}}}
\widetilde{q}_{m,k}\left(  x;J\right)  \widetilde{q}_{m_{k},j}\left(
y;\left[  \varphi_{1},\varphi_{2}\right]  \right)  \left \vert x-\left(
\beta-\alpha \right)  b_{m}\frac{k}{m}-\alpha b_{m}\right \vert \\
& \leq \left[
{\displaystyle \sum \limits_{k=0}^{m}}
{\displaystyle \sum \limits_{j=0}^{m_{k}}}
\widetilde{q}_{m,k}\left(  x;J\right)  \widetilde{q}_{m_{k},j}\left(
y;\left[  \varphi_{1},\varphi_{2}\right]  \right)  \left(  x-\left(
\beta-\alpha \right)  b_{m}\frac{k}{m}-\alpha b_{m}\right)  ^{2}\right]
^{1/2}\\
& =x^{2}\widetilde{\mathcal{B}}_{m}^{Cl}\left[  1;\Delta \right]  -2x\widetilde{\mathcal{B}}_{m}^{Cl}\left[  x;\Delta \right]  +\widetilde{\mathcal{B}}_{m}^{Cl}\left[  x^{2}%
;\Delta \right]  ,
\end{align*}
which converges to zero when $m\rightarrow \infty$ uniformly from the results
in Lemma \ref{lemma2}. Similarly, the second term converges to zero when $m\rightarrow
\infty$ uniformly. Finally, if we get $\delta_{1}=\delta_{2}=\frac{1}{\sqrt
{m}},$ it follows \ $w\left(  \frac{1}{\sqrt{m}},\frac{1}{\sqrt{m}}\right)
\rightarrow0$ when $m\rightarrow \infty.$ Hence, the proof is completed.
\end{proof}

\begin{example}
Let
 \[
{\varphi}_{1}(x) = 0.5 \sin\left(\frac{2 \pi x}{\beta - \alpha}\right)
\]
\[
{\varphi}_2(x) = 0.5 + 0.5 \cos\left(\frac{2 \pi x}{\beta - \alpha}\right)
\]
\[
\widetilde{\varphi}_{1}(u) = {\varphi}_1((\beta - \alpha) ub_{m} + \alpha b_{n} )
\]
\[
\widetilde{\varphi}_{2}(u) = {\varphi}_2((\beta - \alpha) ub_{m}  + \alpha b_{n} )
\]
\[
G(u, v) = \exp\left(-\left[\left(x - 0.5(\alpha + \beta)\right)^2 + \left(y - 0.5\right)^2\right]\right)
\]
\[
x = (\beta - \alpha) u b_{m} + \alpha b_{m}, \quad y = (\widetilde{\varphi}_{2}(u) - \widetilde{\varphi}_{1}(u)) v + \widetilde{\varphi}_{1}(u)
\]

 Figure 3 presents the convergence of the operator $\widetilde{\mathcal{B}}_{m}^{Cl}\left[  G\left(  u,v\right)  ;\Delta \right]$ to the function $G$ for $b_{m}=\sqrt{m}$ and, $m = 50$ and $m=40$, respectively.
 \end{example}
\begin{figure}[ht]
    \centering
    \includegraphics[width=0.45\textwidth]{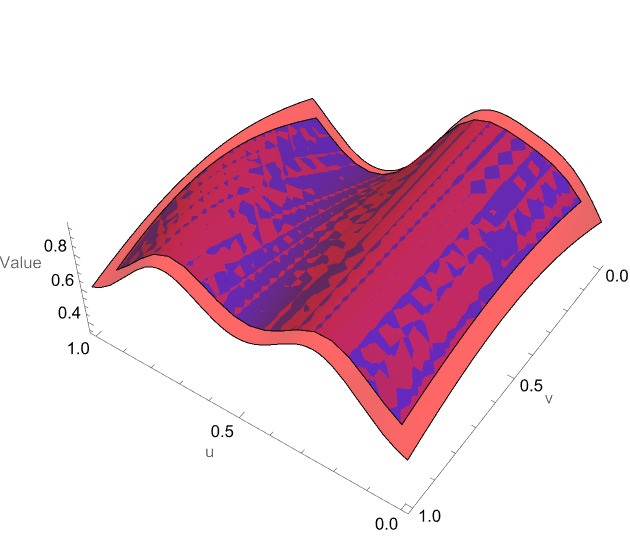}
    \hfill
    \includegraphics[width=0.45\textwidth]{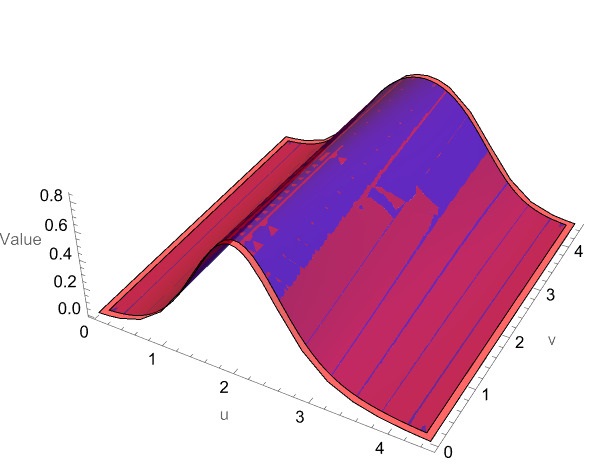}
    \caption{ Approximation of the operator $\widetilde{\mathcal{B}}_{m}^{Cl}\left[  G\left(  u,v\right)  ;\Delta \right]$ for $b_{m}=\sqrt{m}$ and, $m=50$ and $m=40$, respectively.}
    \label{fig:sidebyside}
\end{figure}

\textbf{Acknowledgements} Not applicable.\\

\textbf{Authors' contributions}
Both authors contributed equally to this work. Both authors have read and approved the final manuscript.\\

\textbf{Funding} Not applicable. \\

\textbf{Data availability}
Data sharing is not applicable to this article as no data sets were generated or analyzed during the current study.\\

\section*{Declarations}
\textbf{Conflict of interest} The authors declare no competing interests.\\

\textbf{Ethical Approval} Not applicable.

\end{document}